\documentclass[12pt]{article}
\usepackage{graphicx} 
\usepackage{tikz}
\usepackage{subcaption}
\usepackage{amsmath, amsthm, amssymb}
\usepackage{fullpage}
\usepackage{xcolor}
\usepackage{makeidx}
\usepackage[normalem]{ulem}
\usepackage{caption}
\usetikzlibrary{decorations.pathreplacing}
\makeindex
\theoremstyle{definition}
\newtheorem{theorem}{Theorem}[section]
\newtheorem{lemma}[theorem]{Lemma}
\newtheorem{corollary}[theorem]{Corollary}

\newtheorem{definition}[theorem]{Definition}
\newtheorem{remark}[theorem]{Remark}
\newtheorem{example}[theorem]{Example}
\DeclareMathOperator*{\Res}{Res}
\newcommand{\floor}[1]{\left\lfloor #1 \right\rfloor}

\title{Laplace measure transitions and ghosts\\ for meromorphic functions} 
\author{João Fontinha \thanks{The first author acknowledges partial support from FCT grant 2024.02716.BD.} \\ DCM and CEMS.UL \\ Ciências ULisboa  \\ jmfontinha@ciencias.ulisboa.pt  \and  Jorge Buescu \thanks{The second author acknowledges partial support from CEMS.UL - Center for Mathematical Studies, Ciências, ULisboa
FCT - UID/04561/2025 - https://doi.org/10.54499/UID/04561/2025.}\\ DCM and CEMS.UL \\ Ciências ULisboa \\ jsbuescu@ciencias.ulisboa.pt \and  Jaouen Ramalho \\ DCM \\ Ciências ULisboa \\ fc61993@alunos.ciencias.ulisboa.pt}

\begin{document}
\maketitle
\begingroup
\renewcommand{\thefootnote}{}
\footnotetext[0]{The authors declare no conflict of interest.}
\endgroup
\printindex
\vspace{-0.3em}
\begin{abstract}
We study the measure transition problem for bilateral Laplace transforms of meromorphic functions on vertical strips. Given a meromorphic function \(F\) admitting Laplace representations on two adjacent strips separated by a vertical line, we investigate how the corresponding determining measures are related. Our first result shows  that in the absence of poles on the separatrix the determining measures coincide. We next derive explicit transition formulas for the case of finitely many poles and obtain sufficient conditions under which these formulas remain valid for infinitely many poles. Applications are given to the analytic continuation of the Riemann $\zeta$ function, periodic and almost periodic functions, and quotients of $\Gamma$ functions related to the confluent hypergeometric function. Finally, using generalized Cauchy integrals, we construct an entire function admitting distinct Laplace representations on the right and left half-planes, thereby producing a ghost transition. This  provides a counterexample to uniqueness of solutions of the Cauchy problem for the heat equation.
\end{abstract}
\medskip


\medskip
\vspace{-0.60em}

\section{Introduction}

Let $\mu: \mathbb{R} \longrightarrow \mathbb{C}$ be a function of bounded variation on every finite interval. We consider the  classical bilateral Laplace-Stieltjes transform of $\mu$ in the sense of Widder \cite{Widder}, defined by \begin{equation}\label{eq:Lapdef} \mathcal{L}(\mu)(s)= \int_{- \infty}^{+ \infty} e^{-st} \,d\mu(t). \end{equation}

\noindent Denote by $\Omega_{\alpha,\beta}= \{ s \in \mathbb{C} : \alpha<\Re(s)<\beta  \}$,  where $- \infty \leq \alpha < \beta \leq + \infty$, the vertical strip on the complex plane bounded by $\alpha, \beta$. Classical Laplace transform properties  ensure that $\mathcal{L}(\mu)(s)$ 
is holomorphic in its maximal strip of convergence $\Omega_{a,b}$, with  $- \infty \leq a < b \leq + \infty$.

The measure transition problem for the Laplace transform may be formulated as follows:

\textit{Let $F(s)$ be a meromorphic function on the vertical strip $\Omega_{a,b},$ whose singularities lie on the line $\Re(s)=c,$ for some $a<c<b.$ Suppose that $F$ admits Laplace representations 
\[F(s)=\mathcal{L}(\mu_l)(s) \quad {\rm for} \ s \in \Omega_{a,c} \quad {\rm and}  \quad F(s)=\mathcal{L}(\mu_r)(s) \quad for \quad s \in \Omega_{c,b}.\]  
Establish the relationship between the determining measures $\mu_r$ and $\mu_l.$ }

Several approaches to the measure transition problem have appeared in the literature. 
The first mention of the problem appears to date back to Van der Pol \cite{VanderPol}, who used classical Laplace inversion techniques to derive a transition formula for the difference $\mu_r-\mu_l$ in the case of a finite number of poles. 
His method relied on the uniform vanishing of the quotient  \(F(s)/s\) as \(|\Im(s)| \to +\infty\) inside the strip $\Omega_{a,b}$. The significance of the problem was demonstrated through applications of the transition formula to the analytic continuation of several meromorphic functions of number-theoretic interest, including the Riemann $\zeta$ function.
The subject was later revisited independently by Buescu {\em et al.} \cite{BuescuReal, BuescuComplex, 
BuescuStrips, BuescuPolar, BuescuSpecial} and Harper \cite{Harper}. 
In the first case the context is that of polar (co-positive or co-negative definite, see Definition \ref{def_coco} below) meromorphic functions.
 In this setting, the determining measures are decomposed into principal and holomorphic parts, and a transition formula for $d\mu_r-d\mu_l$ was obtained under the assumption that the holomorphic components extend across the separatrix. 
A different point of view was considered by Harper,  who established a connection between the existence of two Laplace representations for a holomorphic function $F$ on a vertical strip and solutions of the Cauchy problem for the heat equation on the infinite rod. By imposing suitable bounds on the $L^2$-norm of $F$ along vertical lines, Harper applied uniqueness results for the heat equation to obtain sufficient conditions under which $\mu_r=\mu_l$.

This paper is organized as follows. 
In section \ref{sec_definitions} we fix the definitions and notation used throughout. In section \ref{sec_transition_formula} we provide sufficient conditions on the vertical growth of \( F(s) \) ensuring that \( \mu_r = \mu_l \) when no poles are present on the separatrix \( \Re ( s) = c \). The proof combines classical inversion techniques with Phragmén--Lindelöf principles. This allows us to generalize the transition formulas of \cite{VanderPol} and \cite{BuescuPolar} for a finite number of poles. In section \ref{sec_infinite_poles} we provide sufficient conditions under which our transition formula remains valid in the presence of infinitely many poles. 
In Section \ref{sec_applications} we present several examples and applications of the transition formula to problems of analytic continuation, including the zeta function, periodic and almost periodic functions and confluent hypergeometric functions, extending examples previously considered in \cite{BuescuPolar}.
In Section \ref{sec_ghosts}, using generalized Cauchy integrals, we investigate how relaxing the Phragmén--Lindelöf growth conditions allows for the construction of an entire function admitting Laplace representations $F(s)=\mathcal{L}(\mu_r)(s)$ on the right half-plane $\Re(s)>0$ and $F(s)=\mathcal{L}(\mu_l)(s)$ on the left half-plane $\Re(s)<0$, for which the measures $\mu_r$ and $\mu_l$ differ on a subset of the real line of positive Lebesgue measure. This result extends the approach of Harper \cite{Harper} by providing an explicit construction for which $\mu_r \neq \mu_l$, and recovers a well-known counterexample to the uniqueness of the Cauchy problem for the heat equation.

\section{Notation and definitions}
\label{sec_definitions}

\begin{definition}[Normalized determining function]
\label{def_normalized}

\noindent We say that $\mu(t)$ is normalized in $\mathbb{R}$ if and only if $\mu(0)=0$ and \begin{equation*}
    \mu(t)= \frac{\mu(t^+)+ \mu(t^-)}{2} \qquad t \in \mathbb{R}.
\end{equation*}
    
\end{definition}
Throughout this paper all Laplace transforms are taken with respect to normalized determining functions. This is a standard procedure  (see the works of Hille and Widder \cite{ HilleMoment, Widder}) and entails no loss of generality. Indeed, if \( \mu(t) \) is of locally  bounded variation, the set of its points of discontinuity is at most countable. Consequently, normalizing the determining function does not affect the value of the Laplace--Stieltjes integral.

\begin{definition}[Laplace pair]
\label{def_Laplace_pair}

    We say that $F(s)$ is a Laplace pair $(\mu_l, \mu_r)$ on the vertical strip $\Omega_{a,b}$ if $F(s)$ is holomorphic or meromorphic on $\Omega_{a,b}$ having Laplace representations $F=\mathcal{L}(\mu_l)$ and $F=\mathcal{L}(\mu_r)$ on $\Omega_{a,c}$ and $\Omega_{c,b}$ respectively for some $a<c<b.$ We refer to the line $\Re(s)=c$ as a {\em separatrix} for the strip $\Omega_{a,b}.$
\end{definition}

\begin{remark} \label{rem_polos_na_separatriz}
Note that the hypothesis of $F$ having Laplace representations on each of the strips $\Omega_{a,c}$ and $\Omega_{c,b}$ implies that it is holomorphic in each of these strips, so that possible poles must lie on the separatrix $\Re(z)=c$.
   
\end{remark}

\begin{remark}
    The classical bilateral Laplace transform \begin{equation*}
    \mathcal{L}(f)(s)=\int_{-\infty}^{+ \infty}e^{-st}f(t) \,dt
\end{equation*} of a locally integrable  real-valued 
function $f(t)$ may be seen as a special case of \eqref{eq:Lapdef} by setting $\mu(t)= \int_{0}^{t} f(u) \, du.$ Note that this definition implies that $\mu(t)$ is normalized; we sometimes refer to $f(t)$ as {\em normalized determining function} and $(f_l, f_r)$ as {\em Laplace pair}, respectively.  The usage will always be clear from the context.
\end{remark}

As in \cite{BuescuPolar, BuescuSpecial}, to define polar functions on vertical strips we use the following characterization due to Widder
 \cite{Widder2}, which was originally formulated for squares in the complex plane, but can be easily extended to vertical strips.
\begin{definition}
\label{def_coco}
(Co-positive, co-negative definite functions) 
A holomorphic function \(F\) on a vertical strip \(\Omega_{a,b}\) is said to be \emph{co-positive definite} (respectively \emph{co-negative definite}) if it admits an absolutely convergent Laplace representation
\begin{equation}
    \label{eq_Widder}
F(s)=\int_{-\infty}^{+\infty} e^{-st}\,d\mu(t), \qquad s\in\Omega_{a,b},
\end{equation}
where the determining measure \(\mu\) is non-decreasing (respectively non-increasing). We say that \(F\) has \emph{polarity}, or that it is a \emph{polar function} on \(\Omega_{a,b}\), if it is either co-positive definite or co-negative definite on that strip.
    
\end{definition}

\begin{remark}

Note that co-positive definite functions correspond to classical positive definite functions on vertical lines within the strip; along each such line Widder's characterization \eqref{eq_Widder} reduces to Bochner's theorem (see \cite{Reed} for details). 
    
\end{remark}

We shall make  use of the following standard properties of Laplace transforms. 

\begin{lemma}[Inversion theorem for the Laplace transform]\label{lem:inv}

     If $\mu(t)$ is a normalized function of bounded variation in every finite interval, and if the Laplace integral \[ F(s)= \int_{- \infty} ^{+ \infty} e^{-st} \, d\mu(t)\] converges in the strip $\Omega_{a,b}$, then for all $t$ 
     \begin{equation} \label{eq:inv}
        \lim_{R \to \infty} \frac{1}{2 \pi i} \int_{c-iR}^{c+ iR} \frac{F(s)}{s}e^{st} \,ds = \begin{cases} \mu(t)-\mu(- \infty), \quad \text{for} \quad  c>0, \,  a<c<b) \\
        \mu(t)-\mu(+\infty), \quad \text{for} \quad  c<0, \, a<c<b
        \end{cases},
    \end{equation}
    where \( \mu(\pm \infty)\) denote the limits \( \lim_{t \to \pm \infty} \mu(t).\)
 \end{lemma}

 \begin{proof}
     See Widder \cite{Widder}, Chapter VI, Theorem 5b.
 \end{proof}

 \begin{lemma}[Order on vertical lines]
 \label{lem:order}

 Let $F(s)=\int_{-\infty}^{+ \infty}e^{-st} \,d\mu(t)$ have a vertical strip of convergence $\Omega_{a,b}, {\rm with}  - \infty < a<b < + \infty$. For every $\epsilon>0$ and $0<\delta<\frac{|b-a|}{2},$ there exists a positive constant $Y$ such that, for all $a+ \delta \leq x \leq b- \delta$, 
 $\displaystyle{\left| \frac{F(x+iy)}{y}\right|< \epsilon}$
whenever $|y|>Y.$
   
\end{lemma}
\begin{proof}
    See Widder \cite{Widder}, Chapter II, section 13; and Hille \cite{HilleExtremal} for the sharpness of the bound.
\end{proof}

We will need the version of the classical Phragmén-Lindelöf theorem for vertical half-strips.

\begin{theorem}[Phragmén-Lindelöf]\label{thm:lindaloffi} Let $D$ be a vertical half-strip defined by the points of $\Omega_{a,b}$ lying either above or below the line $\Im(s)=d.$ Let $F(s)$ be holomorphic in $D$ and continuous up to its boundary. Assume that \( |F(s) | \leq M\) on boundary lines. If there exist positive constants $A,$ $B$ and $0<K< \frac{\pi}{|b-a|}$ such that for all $s$ inside $D$ we have \[ \label{eq:prePL}|F(x+iy)| \leq Ae^{Be^{K|y|}  } , \] 
then $|F(s)| \leq M$ on $\overline{D}.$

\end{theorem}

\begin{proof}
   See Titchmarsh \cite{TitchmarshFunction}, Chapter V,  or Cartwright \cite{Cartwright}, Chapter III.
\end{proof} 

 \begin{definition}[Phragm\'{e}n-Lindel\"{o}f condition] \label{def:PL}
     
    A function $F(s)$ meromorphic  
     on a vertical strip $\Omega_{a,b}$ and continuous on the boundary is said to satisfy the {\em Phragm\'{e}n-Lindel\"{o}f condition} if there exist positive constants $A,$ $B,$ $Y$ and $0<K<\frac{\pi}{|b-a|}$  such that, for all $a<x<b$,
     \begin{equation}
     \label{eq_PL}
         |F(x+iy)| \leq Ae^{Be^{K|y|}} 
     \end{equation}
     whenever $|y| \geq Y.$
 \end{definition}

\begin{remark}\label{def:containnopoles}
    Notice that the poles of a meromorphic function satisfying the Phragmèn--Lindelöf condition \eqref{eq_PL} must lie in the region defined by $a< \Re(s)<b$ and $|\Im(s)| < Y.$
\end{remark}

\section{The transition formula}
   \label{sec_transition_formula}

From this point onwards we suppose the function \( F(s) \) is represented by a Laplace pair \( (\mu_l,\mu_r) \) on a vertical strip \( \Omega_{a,b} \) with \( a < 0 < b \), where the imaginary axis \( \Re ( s) = 0 \) is the separatrix. This choice is notationally convenient and clearly implies no loss of generality, since a Laplace pair with separatrix \( \Re (s) = c \) reduces to this setting after a translation of the determining measures by the factor \( e^{-ct} \).

We employ a classical Laplace inversion method to derive sufficient conditions ensuring that \( \mu_r(t) = \mu_l(t) \) for all \( t \) when \( F \) has no singularities on the imaginary axis.

\begin{lemma}\label{lem:prehol}
     Let $F(s)$ be a holomorphic Laplace pair $(\mu_l,\mu_r)$ satisfying the Phragm\'{e}n-Lindel\"{o}f condition \eqref{eq_PL} and with no singularities on the separatrix. Let 
          $0<\delta<\min \{|a|, |b|, 1 \}.$ Then there exist positive constants $R_0$ and $M$ such that 
     \begin{equation}
         \label{eq:quotientbound}
              \sup_{|x| \leq \delta} \Bigg|\frac{F(x \pm iR)}{x \pm iR} \Bigg| \leq M    
     \end{equation}
                     for all $R \geq R_0$.
           \end{lemma}

\begin{proof}
   By Lemma \ref{lem:order} it follows that $F(s)/s$ is bounded by some constant $M_{\delta}>0$ on the vertical lines \( |\Re(s)|= \delta. \)
    Since \( F(s)/s\) satisfies the Phragmén--Lindelöf condition, 
    we may choose $R_0>0$ such that the bound \eqref{eq_PL} holds for \(F(s)/s\) whenever $|y| \geq R_0.$ 
    
    Let $\displaystyle{M_0=\sup_{|x| \leq \delta} \Bigg|\frac{F(x \pm iR_0)}{x \pm iR_0} \Bigg|.}$ 
Setting \(M= \max \{M_0, M_{\delta} \},\) it follows that \(F(s)/s\) is bounded by $M$ on the boundary of the vertical half-strips delimited by \(|\Re(s)|< \delta\) and \(|\Im(s)|>R_0\). Applying Theorem \ref{thm:lindaloffi} to \(F(s)/s\) on both regions we obtain \eqref{eq:quotientbound} for all $R \geq R_0,$ as desired.

\end{proof}

\begin{theorem}\label{thm:hol}
    Let $F(s)$ be a holomorphic Laplace pair $(\mu_l,\mu_r)$ satisfying the conditions of Lemma \ref{lem:prehol}.
    Then $\mu_l=\mu_r$ for all  $t \in \mathbb{R}.$
\end{theorem}

\begin{proof}
    
  Fix \( t \in \mathbb{R} \). Let \( R>0 \) and \( 0<\delta<\min\{|a|,|b|\} \), and let \( \Gamma_{R,\delta} \) denote the positively oriented rectangular contour with vertices at \( \pm\delta \pm iR \) depicted in Figure 1 below.

\begin{center}
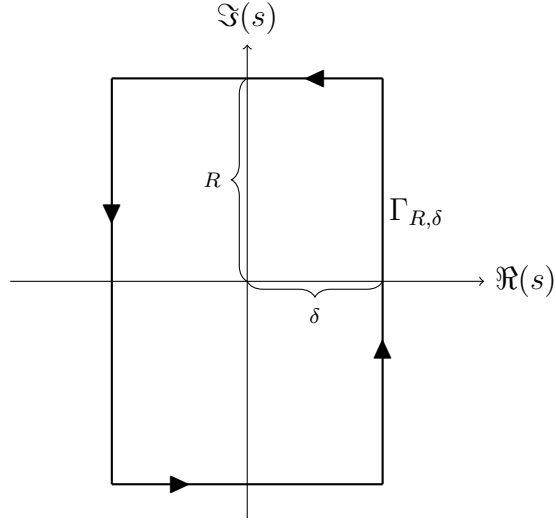

\begin{tikzpicture}[scale=0.9]
  \draw[->] (-3.5,0) -- (3.5,0) node[right] {$\Re(s)$};
  \draw[->] (0,-3.5) -- (0,3.5) node[above]{$\Im(s)$}; 
  \node at (-4,0) {\textcolor{white}{-}};
  \draw[thick] (2,-3) -- (2,3);
  \draw[thick] (2,3) -- (-2,3);
  \draw[thick] (-2,3) -- (-2,-3);
  \draw[thick] (-2,-3) -- (2,-3);
  \draw[decorate,decoration={brace,mirror,amplitude=6pt}] (0,0) -- (2,0)
    node[midway,below=6pt]{\scriptsize$\delta$};
  \draw[decorate,decoration={brace,mirror,amplitude=6pt}] (0,3) -- (0,0)
    node[midway,left=6pt] {\scriptsize$R$};
  \node at (2.5,1) {$\Gamma_{R, \delta}$};
  \node at (2, -1) {$\blacktriangle$}; 
  \node at (1, 3) {\rotatebox{90}{$\blacktriangle$}}; 
  \node at (-2, 1) {\rotatebox[origin=c]{180}{$\blacktriangle$}}; 
  \node at (-1, -3) {\rotatebox[origin=c]{270}{$\blacktriangle$}}; 
\end{tikzpicture}
\captionof{figure}{Contour $\Gamma_{r, \delta}$}
\end{center}

We consider the contour integral $\displaystyle{\frac{1}{2\pi i}\int_{\Gamma_{R,\delta}} e^{st}\frac{F(s)}{s}\,ds.}$
To prove the result it suffices to show that, for a suitable choice of $\delta,$ the contributions of the upper and lower horizontal edges of \( \Gamma_{R,\delta} \) become arbitrarily small for all sufficiently large $R.$ 
Indeed, under this assumption, the inversion formula \eqref{eq:inv} together with the residue theorem yields
\begin{equation}\label{eq:babycontour}
\lim_{R\to+\infty}\frac{1}{2\pi i}\int_{\Gamma_{R,\delta}} e^{st}\frac{F(s)}{s}\,ds
= \mu_r(t)-\mu_l(t)+\mu_l(+\infty)-\mu_r(-\infty)=F(0).
\end{equation}

\noindent Recalling the normalization of the determining measures (definition \ref{def_normalized}), 
evaluation of \eqref{eq:babycontour} at \( t=0 \) shows that
\( \mu_r(-\infty)-\mu_l(+\infty) \) equals the residue contribution $F(0)$. Substituting this back into
\eqref{eq:babycontour} yields \( \mu_r(t)=\mu_l(t) \) for all \( t \), proving the statement of the theorem.

It only remains to show the vanishing of the horizontal integrals, which we do next.

The contribution of each horizontal edge is bounded by
\begin{equation}\label{eq:edgebound}
\left|\int_{-\delta}^{\delta} e^{(x\pm iR)t}\frac{F(x\pm iR)}{x\pm iR}\,dx\right|
\le 2\delta\, e^{\delta|t|} \sup_{ |x| \leq \delta} \Bigg | \frac{F(x \pm iR)}{x \pm iR} \Bigg |.
\end{equation}
Choosing $R_0>0$ as in Lemma \ref{lem:prehol}, it follows that there exists some $M>0$ such that $\displaystyle{\sup_{ |x| \leq \delta} \Bigg | \frac{F(x \pm iR)}{x \pm iR} \Bigg | \leq M}$ whenever $R \geq R_0.$ The right hand side of \eqref{eq:edgebound} becomes bounded by $2 \delta e^{\delta |t|}M.$
Given \( \varepsilon>0 \), choose
\( \delta=\varepsilon/(4e^{h|t|}M) \).
This ensures that the upper and lower integrals \eqref{eq:edgebound} are smaller than \( \varepsilon \) for all sufficiently large $R,$ completing the proof.
\end{proof}

As stated in the Introduction, the inversion argument used by Van der Pol \cite{VanderPol} rested on the assumption that \(F(s)/s\) vanishes uniformly inside the whole strip \(\Omega_{a,b}\) as \( |\Im(s)| \to + \infty.\) The Phragmén-Lindelöf condition \eqref{eq_PL} is the precise analytic mechanism that guarantees that this occurs. As will be shown in Section \ref{sec_ghosts}, if condition \eqref{eq_PL} is relaxed the growth of $F(s)$ on the separatrix may prevent the uniform vanishing of the horizontal edges of the contour, producing examples where $\mu_r \neq \mu_l.$ 

The 
following result generalizes Van der Pol's transition formula to the case of a finite number of poles of arbitrary order.

\begin{corollary}\label{cor:finitepoles}
Let \( F(s) \) be a meromorphic Laplace pair \( (\mu_l,\mu_r) \) having finitely many poles  on the imaginary axis 
\( s=ip_n \), \( 1\le n\le N \), and satisfying the Phragmén--Lindelöf condition \eqref{eq_PL}. Then for each real \( t \)
\begin{equation}
    \label{eq:finite_transition_formula}
\mu_r(t)-\mu_l(t) 
=
\sum_{n=1}^{N}
\Res\!\left((e^{st}-1)\frac{F(s)}{s},\, s=ip_n\right).
\end{equation}

\end{corollary}

\begin{proof}

Proceeding as in the proof of Theorem \ref{thm:hol}, the estimates \eqref{eq:quotientbound} and \eqref{eq:edgebound} remain valid since none of the half-strips contain any poles by Remark \ref{def:containnopoles}. Therefore the same contour argument applies with the only modification that the poles \( s=ip_n \) now lie inside the contour.
By the residue theorem,
\begin{equation}\label{eq:finitepoleseq}
\mu_r(t)-\mu_l(t)+\mu_r(\infty)-\mu_l(\infty)
=
\sum_{n=1}^{N}
\Res\!\left(e^{st}\frac{F(s)}{s},\, s=ip_n\right)
+F(0).
\end{equation}

\noindent Evaluating \eqref{eq:finitepoleseq} at \( t=0 \) and using the normalization of the determining measures, we obtain
\[
\mu_r(-\infty)-\mu_l(+\infty)
=
\sum_{n=1}^{N}
\Res\!\left(\frac{F(s)}{s},\, s=ip_n\right)
+F(0).
\]
Subtracting this identity from \eqref{eq:finitepoleseq} yields \eqref{eq:finite_transition_formula}, as claimed.
\end{proof}

The transition formula \eqref{eq:finite_transition_formula} clarifies the analytic phenomena governing the change of determining measures as a Laplace representation crosses adjacent strips of definition. However, it is generally not convenient for explicit computations, particularly in the presence of poles of high order. Moreover, many of the most relevant examples in practice admit classical Laplace representations by locally integrable functions. For this reason, we derive a more explicit expression for the difference of determining functions in the case where $F(s)$ defines a Laplace pair $(f_l,f_r)$, which is more amenable to computation than \eqref{eq:finite_transition_formula}. It also generalizes the transition formula in \cite{BuescuPolar} which was obtained in the context of polar functions.

\begin{lemma}\label{lem:transitionforfinite}
Let $F(s)$ be a meromorphic Laplace pair $(f_l,f_r)$ satisfying the hypotheses of Corollary~\ref{cor:finitepoles}. 
For each pole $s=ip_n$, let $r_n$ denote its order and let
$\displaystyle{P_n(s)= \sum_{k=1}^{r_n}\frac{a_{n,k}}{(s-ip_n)^k}}$
be the principal part of its Laurent expansion around $s=ip_n$. Then
\begin{equation}\label{eq:transitionformula}
f_r(t)-f_l(t)= \sum_{n=1}^{N}e^{ip_nt}\sum_{k=1}^{r_n}a_{n,k}\frac{t^{k-1}}{(k-1)!}.
\end{equation}
\end{lemma}

\begin{proof}
Since $\mu_{r,l}(t)= \int_{0}^{t} f_{r,l}(u)\,du$, Corollary~\ref{cor:finitepoles} yields
\[
f_r(t)-f_l(t)
=
\frac{d}{dt}
\sum_{n=1}^{N}
\Res\!\left(e^{st}\frac{F(s)}{s};\, s=ip_n\right).
\] We proceed with the construction of an explicit expression for 
$\displaystyle{
\Res\!\left(e^{st}\frac{F(s)}{s};\, s=ip_n\right).}$ Expanding the functions $e^{st}$, $1/s$, and $F(s)$ in Laurent series around $s=ip_n$, we obtain
\begin{equation*}
    \label{eq_finite_transition}
e^{st}\frac{F(s)}{s}
=
e^{itp_n}
\left[
\sum_{\substack{l,m\ge 0 \\ k\ge -r_n}}
\frac{(-1)^m}{(ip_n)^{m+1}}
\frac{t^l}{l!}
b_{n,k}
(s-ip_n)^{m+l+k}
\right],
\end{equation*}

\noindent where $b_{n,k}=a_{n,-k}$ and $-r_n\le k\le -1.$ 

The residue at $s=ip_n$ is the coefficient of $(s-ip_n)^{-1}$, and so
\[
\Res\!\left(e^{st}\frac{F(s)}{s},\, s=ip_n\right)
=
e^{itp_n}
\left[
\sum_{\substack{l,m\ge 0 \\ k\ge -r_n \\ m+l+k=-1}}
\frac{(-1)^m}{(ip_n)^{m+1}}
\frac{t^l}{l!}
b_{n,k}
\right].
\]

\noindent Since $l$ and $m$ are non-negative, the condition $m+l+k=-1$ implies
$0\le l\le -k-1$, and therefore it follows that
\begin{equation}\label{eq:residuecomputation}
\Res\!\left(e^{st}\frac{F(s)}{s},\, s=ip_n\right)
=
e^{itp_n}
\left[
\sum_{k=-r_n}^{-1}
\sum_{l=0}^{-k-1}
b_{n,k}
(-1)^{k+l+1}
(ip_n)^{k+l}
\frac{t^l}{l!}
\right].
\end{equation}

\noindent Differentiating with respect to $t$ we obtain
\begin{equation}\label{eq:differentiationofresidues}
\frac{d}{dt}\Res\!\left(e^{st}\frac{F(s)}{s},\, s=ip_n\right)= e^{itp_n}\sum_{k=-r_n}^{-1}
b_{n,k}
\frac{t^{-k-1}}{(-k-1)!}
=
\sum_{k=1}^{r_n}
a_{n,k}
\frac{t^{k-1}}{(k-1)!}.
\end{equation}
Finally, summing over all poles yields \eqref{eq:transitionformula}, completing the proof
\end{proof}

\begin{remark}[The Mellin Transform] It is possible to rephrase our results in terms of the Mellin transform. 
Performing the change of variables $u=e^{-t}$ in the Laplace transform $\mathcal{L}(\mu)$, we obtain the Mellin transform
\[ \mathcal{M}(\beta)(s)= \int_{0}^{+ \infty} t^s \,d\beta(t), \quad \beta(t)=-\mu(\log(1/t)).\] 
Similarly, for $f \in L^1_{\text{loc}}(\mathbb{R}),$ we obtain $\displaystyle{\mathcal{M}(g)(s)=\int_{0}^{+ \infty}t^{s-1}g(t) \,dt},$ with $g(t)=f(\log(1/t)).$ Since the strip of convergence $\Omega_{a,b}$ remains unchanged, we define a \emph{Mellin pair} $(\beta_l, \beta_r)$ or $(g_l, g_r)$ in a completely analogous manner to Definition \ref{def_Laplace_pair}. Indeed, the Mellin pair $(\beta_l, \beta_r)$ corresponds to the Laplace pair $(\mu_l, \mu_r)$ where $\mu_{r,l}(t)=-\beta_{r,l}(e^{-t})$ while the pair $(g_l, g_r)$ corresponds to $(f_l, f_r)$ with $f_{r,l}(t)=g_{r,l}(e^{-t}).$ This allows us to obtain the transition formulas corresponding to \eqref{eq:finite_transition_formula} and \eqref{eq:transitionformula} in terms of Mellin pairs:
\begin{align*}
    &\beta_r(t)-\beta_l(t)=\mu_l(\log(1/t))-\mu_r(\log(1/t))= \sum_{n=1}^{N}\Res \bigg( (1-t^{-s})\frac{F(s)}{s}, \, s=ip_n \bigg ); \\
&g_r(t)-g_l(t)= \sum_{n=1}^{N}t^{-ip_n}\sum_{k=1}^{r_n}a_{n,k}\frac{(\log(1/t))^{k-1}}{(k-1)!}.
\end{align*} 
\end{remark}

\section{The case of infinitely many poles}
\label{sec_infinite_poles}

In this section we study the generalization of the transition formulas \eqref{eq:finite_transition_formula} and \eqref{eq:transitionformula} for a meromorphic function with an infinite number of poles on the separatrix  (note that, by definition of meromorphy, the poles are isolated and do not have a limit point, forming a discrete set.) Throughout this section we assume $F(s)$ to be a Laplace pair $(\mu_l, \mu_r)$ having a countable set of poles on the imaginary axis at $s=ip_n$ where $\{p_n\}_{n \in \mathbb{N}}$ is a sequence of real numbers with no finite accumulation point.  Let $P_n(s)= \sum_{k=1}^{r_n}\frac{
a_{n,k}}{(s-ip_n)^k}$ denote the principal part expansion around $s=ip_n.$ Then \[P(s)= \sum_{n=1}^{+ \infty}\sum_{k=1}^{r_n}\frac{a_{n,k}}{(s-ip_n)^k}= \sum_{n=1}^{+ \infty}P_n(s) \] denotes the formal series of the principal part expansions around each pole.

Some natural difficulties arise when trying to adapt the inversion argument used in the case of a finite number of poles.
In the first place, unless suitable conditions are imposed on the distribution of the poles and on the growth of the corresponding residues, one cannot control the uniform vanishing of the upper horizontal parts of the contour.
Secondly, the sums appearing in the transition formulas \eqref{eq:finite_transition_formula} and \eqref{eq:transitionformula} are replaced by infinite series, raising non-trivial questions of convergence.

We begin by finding a sufficient condition for the series $P(s)$ to define a meromorphic function represented by a Laplace pair on the right and left half-planes. We denote below by $H(t)$ the Heaviside function.

\begin{lemma}\label{lem:polespart}
Let $P(s)$ denote the formal series defined above and $f(t)$ be its formal inverse Laplace transform given by \begin{equation}\label{eq:formalinverse}
    f(t)=\sum_{n=1}^{+\infty} e^{ip_n t}\sum_{k=1}^{r_n}a_{n,k}\frac{t^{k-1}}{(k-1)!}.
\end{equation} Suppose that for every $\epsilon>0$,
\begin{equation}\label{eq:convergenceinfinitepoles}
\sum_{n=1}^{+\infty}\sum_{k=1}^{r_n}\frac{|a_{n,k}|}{\epsilon^k}<+\infty.
\end{equation}
Then $f(t)$ is a locally integrable function and $P(s)$ defines a Laplace pair \( (f_l, f_r) \) on $\mathbb{C}$ with poles at each $s=ip_n$, where \(f_l(t)=f(t)(H(t)-1)\) and \(f_r(t)= f(t)H(t)\).  

\end{lemma}

\begin{proof}
Let $K$ be a compact subset of $\mathbb{C}$ not containing any of the points $s=ip_n$. 
Since the poles are isolated, only finitely many lie within distance $1$ of $K$. Consequently, there exists a constant $C>0$, depending only on $K$, such that

\[ |P(s)| \leq \sum_{\substack{n,k \\ d(ip_n,K)<1}} \frac{|a_{n,k}|}{|s-ip_n|^k}+ \sum_{\substack{n,k \\ d(ip_n,K) \geq1}} \frac{|a_{n,k}|}{|s-ip_n|^k} < C + \sum_{n=1}^{+\infty}\sum_{k=1}^{r_n}|a_{n,k}|\]

The latter series converges by \eqref{eq:convergenceinfinitepoles} (with $\epsilon=1$). 
Hence $P(s)$ converges uniformly on compact subsets not containing poles 
and, by the Weierstrass theorem, defines a meromorphic function on $\mathbb{C}$ with poles at each $s=ip_n$.

To establish the Laplace representation for $\Re(s)>0$, it suffices to show that for every $\epsilon>0$ there exists $C_\epsilon>0$ such that, for all $t \geq 0$, $|f(t)| \le C_\epsilon e^{\epsilon t}.$

For $t\ge 0$, we have $\frac{t^{k-1}}{(k-1)!} \le \frac{e^{\epsilon t}}{\epsilon^{k-1}},$
and therefore
\begin{equation}\label{eq:funi}
|f(t)|
\le e^{\epsilon t}
\sum_{n=1}^{+\infty}\sum_{k=1}^{r_n}
\frac{|a_{n,k}|}{\epsilon^{k-1}}
\le C_\epsilon e^{\epsilon t},
\end{equation}
where $C_\epsilon<\infty$ by \eqref{eq:convergenceinfinitepoles}.
This bound justifies the interchange of  summation and integration, showing that $f(t)$ is locally integrable. For $\Re(s)>0$,
\[
\int_0^{+\infty} e^{-st}f(t)\,dt
=
\sum_{n=1}^{+\infty}\sum_{k=1}^{r_n}
a_{n,k}
\int_0^{+\infty}
e^{-st}\frac{t^{k-1}}{(k-1)!}\,dt
=
P(s).
\]
The same argument applies for $\Re(s)<0$, completing the proof.
\end{proof}

\begin{theorem}\label{thm:infinitepoles}
Let the poles of $F(s)$ be such that its principal part $P(s)$ satisfies the hypotheses of Lemma~\ref{lem:polespart}. 
If the holomorphic part of $F$,  defined by $H(s)=F(s)-P(s)$,
satisfies the Phragmén--Lindelöf condition \eqref{eq_PL}, then
\begin{equation}\label{eq:infinitetransitionformula}
\mu_r(t)-\mu_l(t)
=
\sum_{n=1}^{+\infty}
\Res\!\left((e^{st}-1)\frac{F(s)}{s},\, s=ip_n\right)
\end{equation} where the series on the right converges uniformly on every compact subset of the real line.

\end{theorem}

\begin{proof}
Let $\mu_{r,P}$ and $\mu_{l,P}$ denote the determining measures of $P(s)$ 
on the right and left half-planes, respectively. Then $H(s)$ defines a holomorphic Laplace pair \( (\mu_{l,H}, \mu_{r,H} )\) where
\( \mu_{r,H}(t)=\mu_r(t)-\mu_{r,P}(t) \) and \( \mu_{l,H}(t)=\mu_l(t)-\mu_{l,P}(t).\)

Since $H(s)=F(s)-P(s)$ is holomorphic and satisfies the Phragmén--Lindelöf condition, it follows from
Theorem~\ref{thm:hol} that \( \mu_{r,H}(t)=\mu_{l,H}(t)\), 
and therefore 
\[
\mu_r(t)-\mu_l(t)
=
\mu_{r,P}(t)-\mu_{l,P}(t).
\]
Then, by Lemma~\ref{lem:polespart},
\[
\mu_{r,P}(t)-\mu_{l,P}(t)
=
\int_0^t f(u)\,du,
\]
where $f(t)$ is given by the series  \eqref{eq:formalinverse}.
Estimate \eqref{eq:funi} now ensures that this series converges uniformly on compact sets, so that integration may be carried out termwise. Using the residue computations \eqref{eq:residuecomputation} and \eqref{eq:differentiationofresidues} in Lemma \ref{lem:transitionforfinite} for the finite pole case, we obtain the transition formula
\eqref{eq:infinitetransitionformula}, as desired.
\end{proof}

\section{Examples and applications}
\label{sec_applications}
\begin{example}{(The zeta function)}
\label{ex_zeta_function}
The  Riemann zeta function is defined by \( \zeta(s)= \sum_{n=1}^{+ \infty} \frac{1}{n^s}\) for $\Re(s)>1$. It is easily shown that its expression  as a Laplace integral of a normalized measure $\mu_{1, \infty}$ in that half-plane is 

\begin{equation}\label{eq:zetamu}
    \zeta(s)= \int_{-\infty}^{+ \infty}e^{-st}\,d\mu_{1, \infty}(t), \quad \mu_{1, \infty}(t)= \begin{cases}
        \floor{e^t}-\frac{1}{2}, \, t>0, \, e^t \notin \mathbb{Z} \\
        n, \, t= \log(n+1), \, n\geq 0 \\
        -\frac{1}{2}, \, t < 0
    \end{cases}.
\end{equation}
On each vertical line $\Re(s)=x,$ the vertical growth of the zeta function satisfies \begin{equation}\label{eq:zetagrowth}
    |\zeta(x+iy)| = O(|y|^A)
\end{equation} where $A$ is a constant which may depend on $x$; see Ivic \cite{Ivic} Chapter VII, section 5, for a detailed treatment. Thus the zeta function satisfies the Phragmén--Lindelöf condition on every vertical strip. It has a simple pole at $s=1$ with residue $1.$ An application of the transition formula \eqref{eq:finite_transition_formula} for the  crossing of the separatrix $\Re(s)=1$ yields \(\mu_{1, \infty}(t)- \mu_l(t)= e^t-1\). We thus obtain 

\begin{equation*}\label{eq:mul}
    \mu_l(t)= \begin{cases}
        \frac{1}{2}- \{e^t\}, \, t>0, \, e^t \notin \mathbb{Z} \\
        n+1-e^t, \, t= \log(n+1), \, n\geq 0 \\
        \frac{1}{2}-e^t, \, t < 0
    \end{cases}
\end{equation*}

A straightforward calculation shows that the maximal strip of convergence of $\mu_l$ is $\Omega_{0,1}$, so we denote it by $\mu_{0,1}$.  
Hence, on this strip, the Laplace representation
\begin{equation}\label{eq:zeroone}
\zeta(s)=\int_{-\infty}^{+ \infty}e^{-st} \,d\mu_{0,1}(t),
\end{equation} 
holds.

\begin{figure}[h]
\centering

\begin{subfigure}[t]{0.48\textwidth}
\centering

\def\ystretch{0.85}

\begin{tikzpicture}[scale=0.865,yscale=\ystretch]

\draw[->] (-4,0)--(3.5,0) node[right] {\footnotesize $t$};
\draw[->] (0,-2.5)--(0,2.5) node[left] {\footnotesize $\mu_{0,1}(t)$};

\draw[dashed] (-4,1.5)--(3.5,1.5) node[right] {\footnotesize $\frac{1}{2}$};
\draw[dashed] (-4,-1.5)--(3.5,-1.5) node[right] {\footnotesize $-\frac{1}{2}$};

\draw[thick,domain=-4:0,samples=200]
plot(\x,{1.5-3*exp(\x)});

\fill (0,0) circle(1.5pt);

\foreach \n in {1,2,3,4}
{
\draw[thick,domain={2*ln(\n)}:{2*ln(\n+1)},samples=200]
plot(\x,{1.5-3*(exp(\x/2)-\n)});
\fill ({2*ln(\n)},0) circle(1.5pt);
}
\end{tikzpicture}

\end{subfigure}
\begin{subfigure}[t]{0.48\textwidth}
\centering
\begin{tikzpicture}[scale=0.9]

\def\a{0.7}
\def\b{1.1}
\def\c{1.4}
\def\d{1.6}
\def\x{0.4}

\draw[->] (-2,0)--(5,0) node[right]{\footnotesize$t$};
\node[below] at (3*\a,0) {\footnotesize$\log 2$};
\node[below] at (3*\b,0) {\footnotesize$\log 3$};
\node[below] at (3*\c,0) {\footnotesize$\log 4$};
\node at (3*\a,0) {\footnotesize$\shortmid$};
\node at (3*\b,0) {\footnotesize$\shortmid$};
\node at (3*\c,0) {\footnotesize$\shortmid$};

\draw[->] (0,-1)--(0,3) node[left]{\footnotesize$\mu_{1,\infty}(t)$};

\node[left] at (0,\x) {\footnotesize$\frac{1}{2}$};
\node[right] at (0,-\x) {\footnotesize$-\frac{1}{2}$};

\draw[thick] (-2,-\x)--(0,-\x);
\draw[thick] (0,\x)--(3*\a,\x);
\draw[thick] (3*\a,2*\x+\x)--(3*\b,2*\x+\x);
\draw[thick] (3*\b,4*\x+\x)--(3*\c,4*\x+\x);
\draw[thick] (3*\c,6*\x+\x)--(3*\d,6*\x+\x);

\fill (0,0) circle(1.5pt);
\fill (3*\a,2*\x) circle(1.5pt);
\fill (3*\b,4*\x) circle(1.5pt);
\fill (3*\c,6*\x) circle(1.5pt);

\end{tikzpicture}
\end{subfigure}

\caption{Graphs of the measures $\mu_{0,1}$ and $\mu_{1,\infty}$.}
\end{figure}
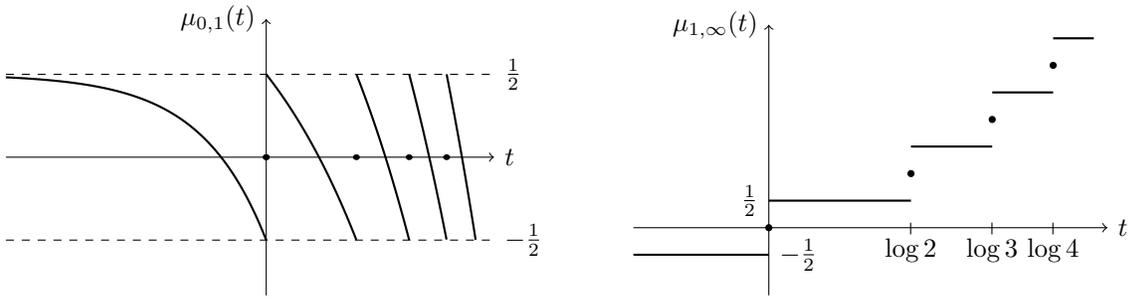

We remark that in \cite{BuescuPolar} (see section 4.5), the authors proved that $\zeta$ must lose its polarity upon crossing the separatrix $\Re(s)=1.$ Our explicit representation on $\Omega_{0,1}$ provides another proof of this fact: since $\mu_{0,1}$ is not monotonic,  the corresponding measure is neither positive nor negative, implying that $F$ in neither co-positive definite nor co-negative definite -- or, equivalently, that is possesses no polarity.

Since there are no poles on the imaginary axis, estimate \eqref{eq:zetagrowth} allows us  to prove that the $\zeta$ function does not have a Laplace representation on any strip \(\Omega_{a,0}\) adjacent to the left of the imaginary axis.
Indeed, if that were the case, direct application of
Theorem \ref{thm:hol} shows that  its determining measure would necessarily be equal to $\mu_{0,1}$. But this is impossible  since $\Omega_{0,1}$ is the maximal strip of convergence of $\mu_{0,1}$. 

However, performing integration by parts once on \eqref{eq:zeroone} we obtain for $s \in \Omega_{0,1}$

\begin{equation}\label{eq:zetasobres01}
    \frac{\zeta(s)}{s}= \int_{-\infty}^{+ \infty}e^{-st}f_{0,1}(t)\,dt, \quad f_{0,1}(t)=\begin{cases}
         - \{e^t \}, \, t \geq 0 \\
    -e^t, \, t<0
    \end{cases}.
\end{equation}

The function \(\frac{\zeta(s)}{s}\) has a simple pole at $s=0$ with residue equal to \(-\frac{1}{2}.\) Applying the transition formula \eqref{eq:transitionformula} yields

\begin{equation}\label{eq:minusonezero}f_{0,1}(t)+\frac{1}{2}= \begin{cases}
    \frac{1}{2}- \{e^t \}, \, t \geq 0 \\
    \frac{1}{2}-e^t, \, t<0
\end{cases}.\end{equation}

It is not difficult to show, with the aid of the substitution \(u=e^t\), that the determining  function obtained in \eqref{eq:minusonezero} has \(\Omega_{-1,0}\) as its maximal strip of convergence. Thus the representation on that strip is equal to

\begin{equation}\label{eq:zetasobresmenosum}
     \frac{\zeta(s)}{s}= \int_{-\infty}^{+ \infty}e^{-st}f_{-1,0}(t) \,dt, \quad f_{-1,0}(t)= \begin{cases}
         \frac{1}{2}- \{e^t \}, \, t \geq 0 \\
    \frac{1}{2}-e^t, \, t<0
    \end{cases}.
\end{equation}

Since there are no poles  on the separatrix $\Re(s)=-1,$ an application of Theorem \ref{thm:hol} once again proves that  \(\frac{\zeta(s)}{s}\) does not have a Laplace representation on any strip \(\Omega_{a,-1}\) adjacent to the left of $\Re(s)=-1.$ 

Observe that \eqref{eq:zetasobres01} and \eqref{eq:zetasobresmenosum} recover the classical Laplace representations of \(\frac{\zeta(s)}{s}\) on successive vertical strips of holomorphy as described in   Titchmarsh \cite{TitchmarshZeta}.

\end{example}

\begin{example}[Periodic  Functions]\label{ex:periodic}

Let $f(t)$ be a bounded locally integrable periodic function of period $T$. 
Its Laplace transform
\[
F(s)=\int_0^{+\infty} e^{-st}f(t)\,dt
\]
is holomorphic for $\Re(s)>0$. By periodicity,
\[
F(s)
=
\sum_{n=0}^{+\infty}
\int_{nT}^{(n+1)T}
e^{-st}f(t)\,dt
=
\sum_{n=0}^{+\infty}
e^{-snT}
\int_0^T e^{-su}f(u)\,du.
\]
Writing
$
{\displaystyle g(s)=\int_0^T e^{-su}f(u)\,du,}
$
we obtain $\displaystyle{F(s)=\frac{g(s)}{1-e^{-sT}}.}$
Thus $F(s)$ extends meromorphically to the complex plane, with simple poles at
$\displaystyle{ s_n=\frac{2\pi i n}{T}, \, n\in\mathbb{Z},}$
whose residues are precisely the Fourier coefficients $a_n$ of $f$.

Since $F(s)$ admits Laplace representations on both half-planes, it defines a Laplace pair $(f_l,f_r)$ satisfying
\[
f_r(t)-f_l(t)=f(t).
\]
On the other hand, the transition formula \eqref{eq:transitionformula} formally yields
\begin{equation}\label{eq:fouriertransition}
f_r(t)-f_l(t)
=
\sum_{n\in\mathbb{Z}}
a_n e^{\frac{2\pi i n}{T}t}.
\end{equation}

Hence, in the periodic setting, the transition formula recovers the Fourier series of $f$. 
Its validity must therefore be interpreted in the same sense in which we take the convergence of the Fourier series to represent the function. We shall use below standard results from the classical theory of convergence and summability of Fourier series which can be found in Katznelson \cite{Katznelson}.

If
$
\sum_{n\in\mathbb{Z}} |a_n| < \infty,
$
the summability condition of Theorem~\ref{thm:infinitepoles} is satisfied. 
Consequently, the transition formula holds pointwise and uniformly, since the Fourier series converges absolutely.

As an illustration of the case where absolute convergence fails, consider
\[
F(s)=\frac{1-e^{-s\pi}}{s(1+e^{-s\pi})}.
\]
The corresponding determining function is the $\pi$-periodic function
\(
f(t)=\operatorname{sgn}(\sin t),
\)
whose Fourier series is given by $\displaystyle{\frac{2}{i\pi}\sum_{n\in\mathbb{Z}}
\frac{1}{2n+1}
e^{i(2n+1)t}.}$
 While this series does not converge absolutely, it is not difficult to show that its Cèsaro means converge to $f$ in $L^1.$ Moreover, this is always true for bounded periodic functions. Therefore, when $f(t)$ is a bounded locally integrable periodic determining function, the transition formula \eqref{eq:fouriertransition} can always be interpreted as converging in $L^1$ via Cèsaro means.

\end{example}
\begin{example}[Stepanoff Almost Periodic Functions]\label{ex:almostperiodic}
A locally integrable function $f$ is called {\em Stepanoff almost periodic} (see \cite{Levitan})
if it is the limit of trigonometric polynomials in the Stepanoff norm
\[
\|f\|_S
=
\sup_{x\in\mathbb{R}}
\int_x^{x+1}|f(t)|\,dt.
\]
For such functions, the Fourier coefficients
\[
a(f,\lambda)
=
\lim_{T\to\infty}
\frac{1}{2T}
\int_{-T}^{T}
f(t)e^{-i\lambda t}\,dt
\]
exist for frequencies $\lambda$ in the spectrum $\sigma(f)$.

Bieberich \cite{Bieberich} proved that every Stepanoff almost periodic function defines a Laplace pair \((f_l, f_r)\) on the complex plane with simple poles at $s=i\lambda$, $\lambda\in\sigma(f)$, and residue $a(f,\lambda),$ where \(f_l(t)=f(t)(H(t)-1)\) and \(f_r(t)=f(t)H(t).\)

Consequently, the transition formula takes the form
\[
f_r(t)-f_l(t)
=
\sum_{\lambda\in\sigma(f)}
a(f,\lambda)e^{i\lambda t}.
\]
 
A classical approximation theorem (see Chapter 2, sections 4-6 of Levitan \cite{Levitan})  ensures the existence of a family of kernels \( \{k_N(s) \}_{n \in \mathbb{N}} \) whose properties imply that, for each positive integer $N,$ the sums
\[
S_N(t)
=
\sum_{\lambda\in\sigma(f)}
k_N(\lambda)
a(f,\lambda)e^{i\lambda t}
\]
have only finitely many nonzero terms and  converge to $f(t)$ in the Stepanoff norm.

Thus, for Stepanoff almost periodic functions, the transition formula is to be interpreted in the Stepanoff topology.

\end{example}

\begin{example}[A quotient of $\Gamma$ functions]\label{ex:gamma}
In \cite{BuescuStrips, BuescuPolar, BuescuSpecial},   the measure transitions for the Gamma function  $\Gamma(s)$ 
were  explicitly computed, revealing that $\Gamma$ alternates between 
co-positive and co-negative definiteness across successive vertical strips of holomorphy.
We now consider a natural generalization and examine how this polarity phenomenon 
behaves for a quotient of Gamma functions.

Let 
\[
\Gamma(a:b;\,s)=\frac{\Gamma(s)\Gamma(a-s)}{\Gamma(b-s)},
\]
where $a,b\in\mathbb{C}$ satisfy $\Re(a)>0$ and $\Re(b)>0$.
This function is related to the confluent hypergeometric function
\[
{}_1F_1(a:b;z)=\sum_{n=0}^{\infty}\frac{(a)_n}{(b)_n}\frac{z^n}{n!}, 
\qquad {\rm where } \ 
(a)_n=\frac{\Gamma(a+n)}{\Gamma(a)},
\]
through the Mellin--Barnes representation
\[
\frac{1}{2\pi i} 
\int_{c-i\infty}^{c+i\infty}
\frac{\Gamma(s)\Gamma(a-s)}{\Gamma(b-s)}
t^{-s}\,ds
=
\frac{\Gamma(a)}{\Gamma(b)}
\,{}_1F_1(a:b;-t),
\]
valid for $0<c<\Re(a)$ and $t>0$; 
see Marichev \cite{Marichev}, Part 1, section 5, or Paris and Kaminski \cite{Paris} section 3.4, for a thorough  treatment.

Since this integral is an inverse Mellin transform, it follows that
\[
\Gamma(a:b;\,s)
=
\int_{-\infty}^{+\infty}
e^{-st}
\frac{\Gamma(a)}{\Gamma(b)}
\,{}_1F_1(a:b;-e^{-t})
\,dt,
\]
which provides a Laplace representation on the strip $\Omega_{0, \Re(a)}$.

Stirling’s formula for the $\Gamma$ function (see  \cite{Paris}, 2.1.8) shows that \[ \Gamma(a:b;\,s)=O \bigg(|y|^{x-1/2+ \Re(a)-\Re(b)}e^{-\frac{\pi}{2}|y|-y \arg(t)} \bigg), \quad |y| \to + \infty \] implying that it satisfies the 
Phragmén--Lindelöf condition on vertical strips, and therefore that the transition formula applies across each pole.

When $a-b\notin\mathbb{Z}$, the function has simple poles at 
$s=-n$ and $s=a+n$, for $n=0,1,2,\dots$. 
The residues at the poles $s=-n$ and $s=a+n$ are given, respectively, by
\[
{\rm Res}(\Gamma(a : b; -n))=\frac{(-1)^n} {n!}\frac{\Gamma(a+n)}{\Gamma(b+n)}
\quad \text{and} \quad {\rm Res}(\Gamma(a : b; a+n))=
\frac{(-1)^n}{n!}\frac{\Gamma(a+n)}{\Gamma(b-a-n)}.
\]

Let $f_0(t)$ denote the determining function on the central strip $\Omega_{0,\Re(a)}$. 
For $n\ge1$, let $f_n(t)$ denote the determining function on the strip 
$\Omega_{a+n-1,a+n}$, 
and let $f_{-n}(t)$ denote the determining function on the strip 
$\Omega_{-n,-n+1}$.
 The transition formula then yields
\[
f_{n+1}(t)-f_n(t)
=
e^{(a+n)t}
\frac{(-1)^n}{n!}
\frac{\Gamma(a+n)}{\Gamma(b-a-n)},
\]
and
\begin{equation}\label{eq:confluenttrans}
f_{-n}(t)-f_{-n-1}(t)
=
e^{-nt}
\frac{(-1)^n}{n!}
\frac{\Gamma(a+n)}{\Gamma(b+n)}.
\end{equation}

In contrast with the $\Gamma$ function, in this case
the residues need not have a fixed sign pattern. 
As \eqref{eq:confluenttrans} shows, the resulting determining functions induce a function $F(s)$ which does not preserve or even possess polarity across adjacent strips.

Suppose now that $b-a=m$ is a positive integer. Then
\begin{equation}\label{eq:quot}
\frac{\Gamma(a-s)}{\Gamma(b-s)}
=
\frac{1}{(a+m-1-s)\cdots(a-s)}.
\end{equation}
Since \[\frac{1}{a+k-s}=\int_{-\infty}^{0}e^{-st}e^{(a+k)t}\,dt, \quad \Re(s)<a+k, \quad 0 \leq k \leq m-1,\] each factor $\frac{1}{a+k-s}$ is co-positive definite on $\Re(s)<a+k$ and co-negative definite on \mbox{ $\Re(s)>a+k.$} It follows that the quotient in \eqref{eq:quot} is co-positive definite on $\Re(s)<a$ and then alternates its polarity upon crossing each strip delimited by the poles at $s=a+k, \,\, 0 \leq m-1.$ 

\begin{center}
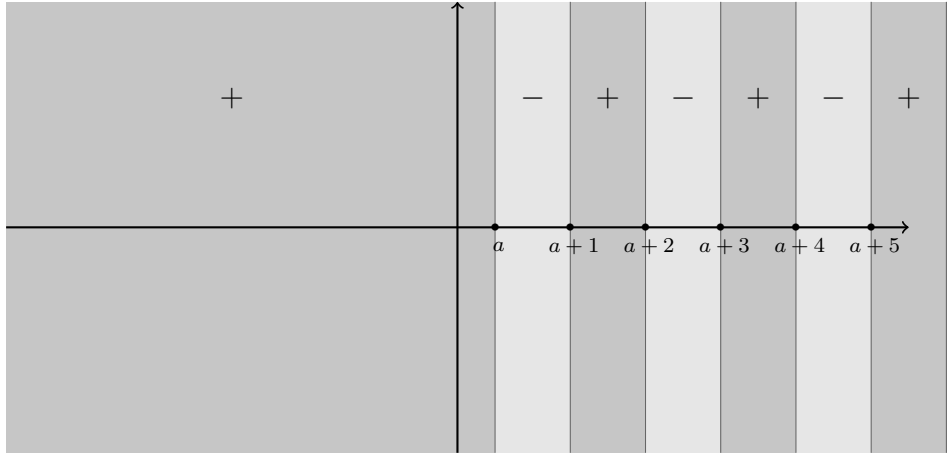

\begin{tikzpicture}[scale=1]
\def\a{0.5}      
\def\xmin{-6}
\def\xmax{6}
\def\ymin{-3}
\def\ymax{3}

\fill[gray!45] (\xmin,\ymin) rectangle (\a,\ymax);

\foreach \k in {0,...,5} {
    \pgfmathtruncatemacro{\parity}{mod(\k,2)}
    \ifnum\parity=0
        \fill[gray!20]
            ({\a+\k},\ymin) rectangle ({\a+\k+1},\ymax);
    \else
        \fill[gray!45]
            ({\a+\k},\ymin) rectangle ({\a+\k+1},\ymax);
    \fi
}

\foreach \k in {0,...,6} {
    \draw[black!60] ({\a+\k},\ymin) -- ({\a+\k},\ymax);
}

\draw[->, thick] (\xmin,0) -- (\xmax,0);
\draw[->, thick] (0,\ymin) -- (0,\ymax);

\foreach \k in {1,...,5}{
  \node[below] at (\a+\k +0.05,0) {\scriptsize$a+\k$};
  \node at (\a+\k,0) {\tiny $\bullet$};
}
\foreach \k in {1,3,5}{
  \node[below] at (\a+\k -0.5,2) {$-$};
}
\foreach \k in {-3,2,4,6}{
  \node[below] at (\a+\k -0.5,2) {$+$};
}
\foreach \k in {0}{
  \node[below] at (\a +0.05,-0.05) {\scriptsize $a$};
  \node at (\a,0) {\tiny $\bullet$};
}
\end{tikzpicture}

\captionof{figure}{Polarity of the quotient \(\frac{\Gamma(a-s)}{\Gamma(b-s)} \) when $b-a$ is an integer.}
\end{center}

Hence, for the special case where $b-a$ is a positive integer, the function $\Gamma(a:b; \, s)$ exhibits the same alternating polarity phenomenon across adjacent strips observed in the Gamma function.

\end{example}

\section{Ghost transitions}
\label{sec_ghosts}

The examples in section \ref{sec_applications} illustrate two types of behavior. For \ref{ex:periodic}, \ref{ex:almostperiodic} and \ref{ex:gamma}, the presence of poles on the separatrix accounted for the difference between $\mu_r$ and $\mu_l$ and the transition formula was valid in the appropriate sense. The second type of behavior entailed the loss of Laplace representation upon crossing a separatrix containing no poles: this is the case of example \ref{ex_zeta_function}, dealing with the $\zeta$ function on the strip $\Omega_{0,1}$ and $\zeta(s)/s$ on $\Omega_{-1,0}.$ This leads to the following question:

\emph{Must analytic continuation beyond a separatrix which contains no poles always be accompanied by the loss of Laplace representability?}

In \cite{Harper}, Harper constructed an example of an entire function having different Laplace representations on strips separated by an arbitrarily small $\epsilon > 0$; however, his construction (based on Newman's entire function bounded on every direction, see \cite{Newman}) requires a nonzero gap between strips which posesses no Laplace representation.
In this section, we answer the question in the negative by constructing an entire function which nonetheless defines a Laplace pair \( (f_l, f_r) \) over \(\mathbb{C} \) (with the imaginary axis as the separatrix) such that the difference \(f_r-f_l\) is non-trivial on some subset of the reals with positive Lebesgue measure.

We call such a phenomenon a {\em Ghost Transition}, since the sudden change in the determining function does not arise from the presence of a singularity in the complex plane.

\begin{definition}[Ghost Transition]

Let $F$ be holomorphic on a vertical strip $\Omega_{a,b}$ defining a Laplace pair $(\mu_l,\mu_r)$ or $( f_l, f_r).$ We say that this Laplace pair is a {\em Ghost Transition} on $\Omega_{a,b}$ if $\mu_r-\mu_l$ (respectively $f_r-f_l$) is nonzero on a set of positive Lebesgue measure.
    
\end{definition}

In view of Theorem \ref{thm:hol}, a ghost transition must correspond to a holomorphic function whose growth along the separatrix violates the Phragm\'{e}n--Lindel\"{o}f condition \eqref{eq_PL}.
Our objective is then to construct holomorphic functions satisfying the growth condition on a vertical line  
\begin{equation}\label{eq:condition}
\limsup_{y\to+\infty}
\frac{|F(iy)|}{\exp\!\big(ae^{y/\epsilon}\big)}
=
+\infty.
\end{equation}
for every $\epsilon>0$ and every $a>0$.

We now construct an explicit example of an entire function satisfying this condition and which is representable by Laplace transforms on the right and left half-planes, showing that ghost transitions do indeed exist. We follow along the lines of Lasse-Rempe \cite{Rempe}; see section 2, which used similar constructions in the context of Transcendental Dynamics.

For each $M>0$ consider a clockwise oriented contour $\gamma_M$ in the complex plane $\mathbb{C}$ where $\zeta=\sigma+i \tau$ where $\gamma_M$ is the union

\begin{equation*}\label{eq:gammaM} \gamma_M= \bigg \{\zeta \in \mathbb{C}: |\sigma|= \frac{\pi}{2 \tau}, \tau\geq M  \bigg \} \cup \bigg \{\zeta \in \mathbb{C}: \tau=M, |\sigma| \leq \frac{\pi}{2M} \bigg \}.\end{equation*}


\noindent Denote by $D_M$ the open region lying to the right of $\gamma_M$ as it is traversed clockwise.

\begin{center}
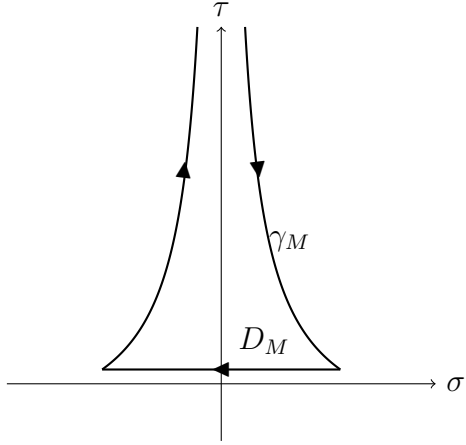

\begin{tikzpicture}[scale=1.9]

\draw[->] (-1.5,0.5) -- (1.5,0.5) node[right] {$\sigma$};
3\draw[->] (0,0.1) -- (0,3) node[above] {$\tau$};

\draw[thick,domain=0.6:3,samples=200]
  plot ({ 1.5/(3*\x) },{\x});

\draw[thick,domain=0.6:3,samples=200]
  plot ({ -1.5/(3*\x) },{\x});

\draw[thick]
  ({-1.5/(3*0.6)},0.6) -- ({1.5/(3*0.6)},0.6);

\node[left] at (0.7,1.5) {$\gamma_M$};
\node at (0.3,0.8) {$D_M$};
\node at (0, 0.6) {\rotatebox{90}{\footnotesize$\blacktriangle$}};
\node at (0.26, 2) {\rotatebox{187}{\footnotesize$\blacktriangle$}};
\node at (-0.26, 2) {\rotatebox{-7}{\footnotesize$\blacktriangle$}};
\end{tikzpicture}
\captionof{figure}{Contour $\gamma_M$ and region $D_M$.}
\end{center}

\begin{lemma}\label{lem:ghostconstruction}

The Cauchy integral \begin{equation*}\label{eq:cauchyint}
I(s)= \frac{1}{2 \pi i}\int_{\gamma_M}\frac{e^{e^{- \zeta^2}}}{s-\zeta} \,d\zeta
\end{equation*} defines a holomorphic function on $\mathbb{C}\setminus \gamma_M,$ and 
\begin{equation}\label{eq:entireghost}
F(s)=
\begin{cases}
I(s), & s\in\mathbb{C}\setminus D_M,\\
e^{e^{-s^2}}+I(s), & s\in D_M.
\end{cases}
\end{equation} extends to an entire function on the complex plane which satisfies the growth condition \eqref{eq:condition}.
    
\end{lemma}

\begin{proof}
    On the contour $\gamma_M$ we have \(\left|e^{e^{-\zeta^2}}\right|
=
\exp\!\big(-e^{\tau^2-\sigma^2}\big) \), and therefore the integral $I(s)$ is locally uniformly convergent and defines holomorphic functions $I_1$ and $I_2$ on $D_M$ and on $\mathbb{C}\setminus \overline{D_M},$ respectively. If $s_0 \in \gamma_M$ we can modify the contour slightly to avoid $s_0,$ and thus see that $I_1$ has analytic continuation to a neighbourhood of $s_0.$ The same is true for $I_2.$
Application of the Cauchy theorem implies that the two extensions are related by
\[
I_2(s_0)= I_1(s_0)+e^{e^{-s_0^2}}.
\]
Therefore, the function defined by \eqref{eq:entireghost} admits an analytic continuation to the whole complex plane which is independent of the initial choice of $M.$ 

Now we prove that $F(s)$ satisfies the growth condition \eqref{eq:condition}. For $y>M,$ we have \(F(iy)=e^{e^{y^2}}+I(iy).\) Since the function $e^{e^{y^2}}$ is easily seen to satisfy \eqref{eq:condition}, our proof will be complete once we show that $I(iy)$ does not interfere with the super-exponential growth.

Denote by $s_1^{\pm}= \pm \frac{\pi}{2 y_1}+iy_1$ the points on $\gamma_M$ closest to $s=iy.$ By considering the isosceles triangle with vertices at $s$ and $s_1^{\pm}$ it follows that $d \geq \frac{\pi}{2y_1},$ where \(d= \inf_{\zeta \in \gamma_M} |\zeta-iy|.\) Since \( \lim_{y \to + \infty} |y-y_1|=0\) uniformly, there is a positive constant $C_1$ such that $C_1 y \geq y_1$ for all pairs \((y,y_1).\) It follows that,  for all $\zeta \in \gamma_M$, 
    \[ |\zeta- iy| \geq d \geq \frac{\pi}{2y_1} \geq \frac{\pi}{2C_1y},\] 
    and therefore
    \begin{equation}\label{eq:Ibound}
        |I(iy)| \leq \frac{1}{2 \pi}\int_{\gamma_M}\Bigg |\frac{e^{e^{-\zeta^2}}}{\zeta-iy}\Bigg | \,d\zeta \leq Cy, \quad C= \frac{C_1}{\pi^2}\int_{\gamma_M} \bigg | e^{e^{-\zeta~2}} \bigg| \,d\zeta< + \infty. 
    \end{equation}
    From \eqref{eq:Ibound} it follows that \(F(iy) \sim e^{e^{y^2}} \) as $y \to + \infty$, satisfying \eqref{eq:condition} as desired.
\end{proof}

\begin{remark}
  The construction of entire functions with prescribed growth properties in specified regions via generalized Cauchy integrals appears to originate in work of Morimoto and Yoshino \cite{Morimoto}. The method was later systematized by Kaneko \cite{Kaneko} who used it to provide explicit examples of Fourier hyperfunctions with support at infinity. Variants of this construction have also appeared in other areas of Complex Analysis. In particular, Eremenko and Goldberg \cite{Eremenko} employed related techniques to disprove a conjecture of Hayman and Erdős concerning asymptotic curves of entire functions. More recently, Rempe \cite{Rempe} used similar constructions in the study of entire functions with pathological dynamical behavior.
\end{remark}

\begin{theorem}\label{thm:ghost}
The entire function $F$ defined in \eqref{eq:entireghost} 
induces a Laplace pair $(f_l,f_r)$ over $\mathbb{C}$ with separatrix  $\Re(s)=0.$ This Laplace pair is a ghost transition with $f_r(t)-f_l(t)=f(t),$ where $f$ is a smooth function of $t$ given by 

\[f(t)=\frac{1}{2\pi i}\int_{\gamma_M}e^{e^{- \zeta^2}}e^{\zeta t} \,d\zeta.\]

\end{theorem}

\begin{proof}
    
We use the Laplace representation of the Cauchy kernel
\[
\frac{1}{s-\zeta}
=
\int_{0}^{+ \infty}e^{-st}e^{\zeta t} \,dt,
\] 
and, for $\Re(s)> \frac{\pi}{2M}$, express $F(s)$ by
\begin{equation*}\label{eq:beforefubini}
    F(s)= \frac{1}{2 \pi i}\int_{\gamma_M}e^{e^{-\zeta^2}}\int_{0}^{+ \infty}e^{-st} e^{\zeta t} \,dt \,d\zeta.
\end{equation*} The super-exponential decay of $e^{e^{-\zeta^2}}$ along $\gamma_M$ implies that the integrand \(e^{e^{- \zeta^2}}e^{-(s-\zeta)t}\) is absolutely integrable on \((\zeta, t) \in \gamma_M \times (0,+ \infty).\) 
Fubini's theorem thus allows the interchange of the order of integration, yielding the Laplace representation  
\begin{equation*}\label{eq:pi2M}
F(s)= \int_{0}^{+ \infty}e^{-st} \Bigg ( \frac{1}{2 \pi i} \int_{\gamma_M} e^{e^{- \zeta^2}}e^{\zeta t} \,d \zeta \Bigg )\,dt, \quad \Re(s)> \frac{\pi}{2M}.
\end{equation*}
Letting $M \to +\infty$ , we conclude that the representation extends up to its abscissa of convergence at $\Re(s)=0.$ The integrand $ e^{e^{- \zeta^2}}e^{\zeta t}$ is an analytic function of $\zeta$, and by Cauchy's theorem it follows that the integral $\displaystyle{f(t)=\frac{1}{2 \pi i} \int_{\gamma_M} e^{e^{- \zeta^2}}e^{\zeta t} \,d \zeta}$ is independent of $M$ and defines a smooth function of $t.$
Thus $F(s)$ admits a Laplace representation for $\Re(s)>0$ with determining function
\[
f_r(t)=  \left( \frac{1}{2 \pi i} \int_{\gamma_M} e^{e^{- \zeta^2}}e^{\zeta t} \,d\zeta  \right)H(t).
\]
Similarly, for $\Re(s)<0$ one obtains the determining function
\[
f_l(t)= \left(\frac{1}{2 \pi i} \int_{\gamma_M} e^{e^{- \zeta^2}}e^{\zeta t} \,d\zeta \right)(H(t)-1).
\]
Hence
\[
f_r(t)-f_l(t)= \frac{1}{2 \pi i} \int_{\gamma_M}e^{e^{-\zeta^2}}e^{\zeta t} \,d\zeta \equiv f(t).
\]

We now show that $f(t)$ is not identically zero. Differentiating three times at $t=0$ and performing the changes of variables $w=-\zeta^2$ and $u=e^w$, we obtain
 \begin{align*}f^{(3)}(0)&=\frac{1}{2 \pi i} \int_{\gamma_M}e^{e^{- \zeta^2}}\zeta^3 \,d\zeta \\
        & = \frac{1}{2}\frac{1}{2\pi i}\int_{\partial_{\pi}}e^{e^{w}}w \,dw= \frac{1}{2}\frac{1}{2 \pi i}\int_{\mathcal{H}}\frac{e^u}{u}\log(u) \,du \\& = \frac{1}{2}\bigg(\frac{1}{\Gamma(1)} \bigg)'=\frac{\gamma}{2}\end{align*} where $\partial_\pi$ denotes the right half-strip delimited by $\Re(s)=0$ and $|\Im(s)|< \pi,$ \(\mathcal{H}\) is the classical Hankel contour for the $\Gamma$ function (see section 13.2.4. of Krantz \cite{Krantz}) and $\gamma$ is the Euler-Mascheroni constant. 
Therefore $f(t) \equiv f_r(t)-f_l(t)$ is not identically zero and is non-zero on a set of positive Lebesgue measure. This concludes the proof.

\end{proof}

\subsection{Connection to the heat equation}

Let \( H: \mathbb{R} \times [0, + \infty) \longrightarrow \mathbb{C}\) be a smooth function of the real variables $(x,t).$ We say that $H(x,t)$ is a \emph{nulltemperature function} if it solves the Cauchy initial value problem for the complex heat equation

\begin{align}\label{eq:cauchy}
\frac{\partial H}{\partial t} = \frac{\partial^2H}{\partial x^2}, \qquad H(x,0)=0. \end{align} 

It has been known since Tikhonov \cite{Tikhonov} that unless suitable growth conditions are imposed on $H(x,t),$ the heat IVP \eqref{eq:cauchy} admits non-trivial solutions. Since then,  considerable effort has been devoted to identifying growth conditions under which the trivial solution is the only solution to \eqref{eq:cauchy}. For a modern survey of results, see \cite{Ferretti}.

In \cite{Harper}, using new Paley--Wiener theorems and analytic continuation techniques for holomorphic functions on strips, Harper proved the following result, which connects the measure transition problem to the theory of the heat equation.

\begin{theorem}[Harper]\label{thm:Harper}

Let \(F(s)\) be a holomorphic function on the vertical strip \(\Omega_{-1,1}.\) If for each $0<|x|<1$ we have 

\begin{equation}\label{eq:l2norm}
    \int_{-\infty}^{+ \infty} |F(x+iy)|^2 \,dy < + \infty,
\end{equation} then $F$ has Laplace representations $F(s)= \mathcal{L}(f_l)(s)$ and $F(s)= \mathcal{L}(f_r)(s)$ valid on $\Omega_{-1,0}$ and $\Omega_{0,1},$ respectively. Moreover, the function given by \begin{equation}\label{eq:fundacalor} H(x,t)=\int_{-\infty}^{+ \infty} e^{-t \lambda^2}e^{ix \lambda}f(\lambda) \,d\lambda, \quad f(\lambda)=f_r(\lambda)-f_l(\lambda) \end{equation} is a nulltemperature function.
    
\end{theorem}

By imposing further restrictions on the vertical $L^2$ norm \eqref{eq:l2norm}, Harper was able to provide sufficient conditions under which ghost transitions do not exist. That is, under such conditions, uniqueness bounds for the heat equation may be applied to conclude that $H \equiv 0.$ This implies that $f_r=f_l$ almost everywhere via Fourier uniqueness.

It is then natural to ask whether the difference \[f_r(t)-f_l(t)=f(t)=\frac{1}{2\pi i}\int_{\gamma_M}e^{e^{- \zeta^2}}e^{\zeta t} \,d\zeta \] obtained from the ghost transition in Theorem \ref{thm:ghost} can be used to construct a non-trivial nulltemperature function via \eqref{eq:fundacalor}.

\begin{theorem}\label{thm:moraldahistoria}
    The function $H: \mathbb{R} \times [0, + \infty) \to \mathbb{C}$ given by
    \begin{equation}
    \label{eq_heat_solution}
        H(x,t)=\int_{-\infty}^{+ \infty} e^{-t \lambda^2}e^{ix \lambda}f(\lambda) \,d\lambda,
    \end{equation} where \( f(\lambda)=\frac{1}{2\pi i}\int_{\gamma_M}e^{e^{- \zeta^2}}e^{\zeta \lambda} \,d\zeta, \) is a non-trivial nulltemperature function.
\end{theorem}

\begin{proof}
Interchanging the order of integration in \eqref{eq_heat_solution} yields

\begin{equation*}\label{eq:heatnonunique} H(x,t)= \frac{1}{2 \pi i}\int_{\gamma_M}e^{e^{-\zeta^2}} \int_{-\infty}^{+ \infty}e^{ix \lambda}e^{-t\lambda^2+\zeta \lambda} \,d\lambda \,d\zeta  = \frac{1}{\sqrt{4 \pi t}i} \int_{\gamma_M} e^{e^{-\zeta^2}}e^{-\frac{(x-i\zeta)^2}{4t}} \,d\zeta . \end{equation*}
Performing the change of variable $i \zeta= \omega$  we recover the known non-trivial solution to the heat IVP \eqref{eq:cauchy}

\[H(x,t)=\frac{1}{\sqrt{4 \pi t}} \int_{-i \gamma_M} e^{e^{\omega^2}}e^{- \frac{(x- \omega)^2}{4t}} \,d\omega \] originally constructed by Chung in \cite{Chung}.
\end{proof}

\end{document}